\documentclass[11pt,letterpaper]{amsart}

\title[Solitons and minimal hypersurfaces]{Soliton solutions of the mean curvature\\ flow and minimal hypersurfaces}

\author[N. Hungerb\"uhler \& T. Mettler]{Norbert Hungerb\"uhler \& Thomas Mettler}

\date{February 10, 2011}

\subjclass[2010]{49Q05}

\thanks{Research for this article was carried out while the authors were supported by the Swiss National Science Foundation, the first author by the grant 200020-124668  and the second by the postdoctoral fellowship PBFRP2-133545.}

\keywords{mean curvature flow, soliton solutions, minimal hypersurfaces, \MA systems, equivalence problem}

\usepackage{graphicx}

\newtheorem{theorem}{Theorem}[section]
\newtheorem*{ct}{Theorem 2.3}
\newtheorem{lemma}[theorem]{Lemma}
\newtheorem{proposition}[theorem]{Proposition}
\theoremstyle{definition}
\newtheorem*{remark}{Remark}
\newtheorem*{definition}{Definition}

\newcommand{\D}{D}
\renewcommand{\d}{\mathrm{d}}
\newcommand{\R}{\mathbb{R}}
\newcommand{\Y}{\mathbf{Y}}
\newcommand{\X}{\mathbf{X}}
\renewcommand{\H}{\mathbf{H}}
\newcommand{\MA}{Monge-Amp\`ere }
\newcommand{\W}{\mathbf{W}}
\newcommand{\0}{\mathbf{0}}

\numberwithin{equation}{section}

\begin{document}

\maketitle

\begin{abstract}
Let $(M,g)$ be an oriented Riemannian manifold of dimension at least $3$ and $\X \in \mathfrak{X}(M)$ a vector field. We show that the \MA differential system (M.A.S.) for $\X$-pseudosoliton hypersurfaces on $(M,g)$ is equivalent to the minimal hypersurface M.A.S. on $(M,\bar{g})$ for some Riemannian metric $\bar{g}$, if and only if $\X$ is the gradient of a function $u$, in which case $\bar{g}=e^{-2u}g$. Counterexamples to this equivalence for surfaces are also given. 
\end{abstract}

\section{Introduction}

Recall that a smooth family of hypersurfaces $F_t : \Sigma^n \to M^{n+1}$, $t \geq 0$, in a Riemannian manifold $(M,g)$ is called a solution of the mean curvature flow (M.C.F.) on $(0,T)$, $T>0$, if
$$
\aligned
\frac{d}{dt} F_t&=-\H, \quad &\textrm{on}&\; \Sigma \times (0,T),\\
F_0&=f, \quad &\textrm{on}&\; \Sigma,\\
\endaligned
$$
where $f : \Sigma \to M$ is a given initial hypersurface and $\H$ denotes the mean curvature vector of $F_t(\Sigma)$. Suppose there exists a conformal Killing vector field $\X$ on $M$ with flow $\varphi : M \times \R \to M$. A family of hypersurfaces $F_t$ is said to be a soliton solution of the M.C.F. with respect to the conformal Killing vector field $\X$ if $\tilde{F}_t=\varphi^{-1}(F_t,t)$ is stationary in normal direction, i.e.~$\tilde{F}_t(\Sigma)$ is the fixed hypersurface $f(\Sigma)$. In~\cite{MR1787070} it was shown that for a given initial hypersurface $f: \Sigma \to M$ to give rise to a soliton solution of the mean curvature flow it is necessary that
\begin{equation}\label{soli}
\H+\X^{\perp}=0,
\end{equation}
where $\perp$ denotes the $g$-orthogonal projection onto the normal bundle of the hypersurface $f : \Sigma \to M$. If $\X$ is Killing, then \eqref{soli} is also sufficient. 

Soliton solutions  have played an important r\^ole  in the development
of  the  theory  of  the  M.C.F.   Such  solutions  served,  e.g.,  as
tailor-made  comparison solutions  to investigate  the  development of
singularities  (e.g.~Angenent's   self-similarly  shrinking  doughnut,
see~\cite{MR1167827}).  Actually, soliton  solutions appear as blow-up
of  so  called type  II  singularities of  the  flow  of plane  curves
(see~\cite{MR1100205}).  Moreover, soliton solutions turn out to enjoy
certain stability properties and allow some insight into the behaviour
of   the  mean   curvature   flow  viewed   as   a  dynamical   system
(see~\cite{MR1787070}, \cite{MR1855161} and~\cite{MR2321890}).

In~\cite{MR1787070}  the boundary value  problem for  rotating soliton
solutions has been discussed. The corresponding local existence result
has been  generalised to arbitrary  Killing fields in~\cite{nhbr2009}.
For  rotating solitons  in  the euclidean  plane,  so called  yin-yang
curves,  a quantity was  identified that  remains invariant  along the
curve  (see~\cite{nhbr2009}).  This  invariant  allowed to  show  that
yin-yang curves  share fundamental geometric  properties with geodesic
curves.   In~\cite{nhbr2009}  the   corresponding  results  have  been
generalised   to   arbitrary   soliton   curves   on   surfaces   (see
Figure~\ref{fig-yinyang}).
\begin{figure}[h!]
\includegraphics[width=.5\linewidth]{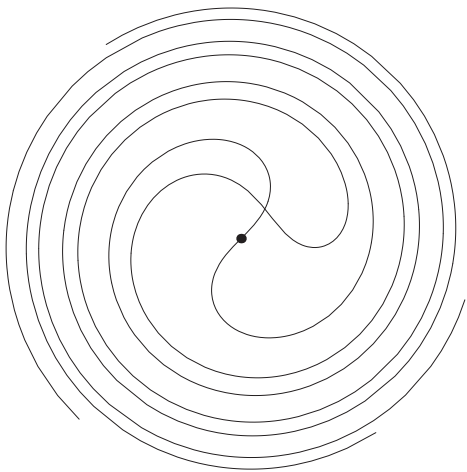}
\caption{If  the Gaussian  curvature of  the simply  connected ambient
  surface  is  less  than or  equal  to  0,  then two  soliton  curves
  intersect in  at most one point.   This fact is  illustrated here by
  two yin-yang curves rotating about the origin.}\label{fig-yinyang}
\end{figure}
In  addition,  it was  observed  in~\cite{nhbr2009}, that  translating
solitons in  the euclidean  plane, the so  called grim  reaper curves,
actually  are  geodesics  with   respect  to  a  conformally  deformed
Riemannian  metric.   Therefore  the  natural question  arose  whether
soliton curves  are (at least locally) {\em  always\/} geodesic curves
with  respect  to  a   modified  Riemannian  metric. This is not the case. On a surface $(M,g)$, the solutions of \eqref{soli} are immersed curves on $M$ which may be reparametrised to become geodesics of the Weyl connection $\nabla_{g,\X}$ given by
$$
(\Y_1,\Y_2)\mapsto (D_g)_{\Y_1} \Y_2-g(\Y_1,\Y_2)\X+g(\X,\Y_1)\Y_2+g(\X,\Y_2)\Y_1,
$$
where we have written $D_g$ for the Levi-Civita connection of $g$. The equation \eqref{soli} is parametrisation invariant and thus its solutions are naturally interpreted as the geodesics of a projective structure on $M$. Recall that a projective structure is an equivalence class of affine torsion-free connections, where two such connections are said to be equivalent if they have the same geodesics up to parametrisation. Recently in~\cite{MR2581355}, Bryant, Dunajski and Eastwood determined the necessary and sufficient local conditions for an affine torsion-free connection to be projectively equivalent to a Levi-Civita connection. Applying their results\footnote{Since the computations are somewhat complex, they have been carried out using maple. The maple file can be obtained from the authors upon request.} it follows that the Weyl connection whose geodesics are the yin-yang curves is not projectively equivalent to a Levi-Civita connection. However J\"urgen  Moser conjectured\footnote{Stated on the occasion of a seminar talk of the first author
at  the Institute  for Mathematical  Research (FIM)  at  ETH Z\"urich,
March 1999.} that  soliton  curves  can  at least  locally be  interpreted  as
geodesics of a Finsler metric.  Recent results about Finsler metrisability of path geometries by 
\'Alvarez  Paiva   and  Berck~\cite{alvarezberckfinslermetri} show that this is indeed the case. Of
course,  one can ask  analogue questions  also for  higher dimensional
solitons.   Before we  do that,  we generalise  the notion  of soliton
solutions slightly.
\begin{definition}
A hypersurface $f : \Sigma \to M$ solving \eqref{soli} for some vector field $\X \in \mathfrak{X}(M)$ will be called a $\X$-\textit{pseudosoliton hypersurface} of $(M,g)$.
\end{definition} Note that the $\0$-pseudosoliton hypersurfaces are the minimal hypersurfaces of $(M,g)$. It was observed in~\cite{MR1855161} (see also~\cite{MR1030675}) that solitons with respect to gradient vector fields correspond to minimal hypersurfaces. However it was left open if such a correspondence holds when the vector field is not the gradient of a smooth function. In this short article we provide an answer using the framework of \MA differential systems. 

In \S2 we will associate to the $\X$-pseudosoliton hypersurface equation on $(M,g)$ a \MA system on the unit tangent bundle of $M$ whose Legendre integral manifolds, which satisfy a natural transversality condition, locally correspond to $\X$-pseudosoliton hypersurfaces on $M$. We then show that for a gradient vector field $\X=\nabla_g u$ on $M$, the $\X$-pseudosoliton M.A.S. is equivalent to the minimal hypersurface M.A.S. on $(M,e^{-2u}g)$. This was already shown in~\cite{MR1855161}, albeit expressed in different language. We complete the picture by proving the
\begin{ct}
The $\X$-pseudosoliton M.A.S. on an oriented Riemannian manifold $(M,g)$ of dimension $n+1\geq 3$ is equivalent to a minimal hypersurface M.A.S. if and only if $\X$ is a gradient vector field. 
\end{ct} 
Theorem \ref{main} is wrong for $n=1$, i.e.~the case of curves on surfaces. We provide counterexamples and comment on the necessary and sufficient conditions for $\X$ in the surface case. Theorem \ref{main} provides an answer to the equivalence problem for specific M.A.S. in arbitrary dimension $n+1\geq 3$. The equivalence problem for general M.A.S. has been studied for $5$-dimensional contact manifolds in~\cite{MR1985469} and in various low dimensions in~\cite{MR1222276}.  

Throughout the article all manifolds are assumed to be connected and smoothness, i.e.~infinite differentiability is assumed. 

\subsection*{Acknowledgements} The second author is grateful to Robert Bryant for helpful discussions. 

\section{Equivalence of the soliton and minimal hypersurface equation}

\subsection{\MA systems}
Let $N$ be a $(2n+1)$-dimensional manifold carrying a contact structure, meaning a maximally nonintegrable codimension 1 subbundle $\D \subset TN$ which we assume to be given by the kernel of a globally defined contact form $\theta$. Recall that a $n$-dimensional submanifold $f: \Sigma \to N$ which satisfies $f^*\theta=0$ is called a Legendre submanifold of $(N,\D)$. A Monge-Amp\`ere differential system on $(N,\D)$ is a differential ideal $\mathcal{M} \subset \mathcal{A}^*(N)$ in the exterior algebra of differential forms on $N$ given by
$$
\mathcal{M}=\left\{\theta,\d\theta,\varphi\right\},
$$ 
where $\varphi \in \mathcal{A}^n(N)$ is a $n$-form.\footnote{More generally one can define a M.A.S. to be a differential ideal which is only locally generated by a contact ideal and an $n$-form. However for our purposes the above definition is sufficient.} The brackets $\{\;\}$ denote the algebraic span of the elements within, i.e.~the elements of $\mathcal{M}$ may be written as
$$
\alpha\wedge\theta+\beta\wedge \d\theta+\gamma\wedge \varphi,
$$
where $\alpha,\beta,\gamma$ are differential forms on $N$. Note that $\mathcal{M}$ is indeed a differential ideal since $\d\varphi$ lies in the contact ideal $\mathcal{C}=\left\{\theta,\d\theta\right\}$, cf.~\cite{MR1985469}. A Legendre submanifold of $(N,\D)$ which pulls-back to $0$ the $n$-form $\varphi$ as well will be called a Legendre integral manifold of $\mathcal{M}$. Two Monge-Amp\`ere systems $(N,\mathcal{M})$ and $(\bar{N},\bar{\mathcal{M}})$ are called equivalent if there exists a diffeomorphism $\psi : N \to \bar{N}$ identifying the two ideals. Note that this implies that $\psi$ is a contact diffeomorphism. 

\subsection{Minimal hypersurfaces via frames}
In order to fix notation we review the description of minimal hypersurfaces using moving frames. For $n \geq 1$, let $(M,g)$ be an oriented Riemannian $(n+1)$-manifold, $\pi : F \to M$ its right principal $SO(n+1)$-bundle of positively oriented orthonormal frames and $\tau : U \to M$ its (sphere) bundle of unit tangent vectors. Write the elements of $F$ as $(p,e_0,\ldots,e_n)$ where $p \in M$ and $e_0,\ldots,e_n$ is a positively oriented $g$-orthonormal basis of $T_pM$. The Lie group $SO(n+1)$ acts smoothly from the right by
$$
(p,e_0,\ldots,e_n)\cdot r=\left(p,\sum_{i=0}^n e_i r_{i0},\dots, \sum_{i=0}^n e_i r_{in}\right),
$$
where $r_{ik}$ for $i,k=0,\ldots,n$ denote the entries of the matrix $r$. The map $\nu : F \to U$, given by $(p,e_0,\ldots,e_n) \mapsto (p,e_0)$ is a smooth surjection whose fibres are the $SO(n)$-orbits and thus makes $F$ together with its right action into a $SO(n)$-bundle over $U$. Here we embed $SO(n)$ as the Lie subgroup of $SO(n+1)$ given by 
$$
\left\{\left(\begin{array}{cc} 1 & 0 \\ 0 & r \end{array}\right) \in SO(n+1), r \in SO(n) \right\}.
$$
Let $\omega_i \in \mathcal{A}^1(F)$ denote the tautological forms of $F$ given by 
$$
\left(\omega_i\right)_{(p,e_0,\ldots,e_n)}(\xi)=g_p\left(e_i,\pi^{\prime}(\xi)\right), 
$$
and $\omega_{ik} \in \mathcal{A}^1(F)$ the Levi-Civita connection forms which satisfy $\omega_{ik}+\omega_{ki}=0$. The dual vector fields to the coframing $\left(\omega_i,\omega_{ik}\right), i < k$, will be denoted by $\left(\W_i,\W_{ik}\right)$. Recall that we have the structure equations
\begin{equation}\label{struceq}
\aligned
\d\omega_i+\sum_{k=0}^n\omega_{ik}\wedge\omega_k&=0,\\
\d\omega_{ik}+\sum_{l=0}^n\omega_{il}\wedge\omega_{lk}&=\Omega_{ik},
\endaligned
\end{equation}
where $\Omega_{ik} \in \mathcal{A}^2(F)$ are the curvature forms. Denote by $\hat{\omega}_i$ the wedge product of the forms $\omega_1,\ldots \omega_n$, with the $i$-th form omitted
$$
\hat{\omega}_i=\omega_1\wedge\cdots\wedge\omega_{i-1}\wedge\omega_{i+1}\wedge \cdots \wedge \omega_n.
$$
For $n=1$ set $\hat{\omega}_1\equiv1.$ Note that the forms 
$$
\aligned
\theta&=\omega_0,\\ 
\omega&=\omega_1\wedge \cdots \wedge \omega_n,\\ 
\mu&=-\frac{1}{n}\sum_{i=1}^n (-1)^{i-1}\omega_{0i} \wedge \hat{\omega}_i,\\
\endaligned
$$
are $\nu$-basic, i.e.~pullbacks of forms on $U$ which, by abuse of language, will also be denoted by $\theta,\omega,\mu$. Since
\begin{equation}\label{contact}
\d\omega_0=-\sum_{k=1}^n\omega_{0k}\wedge \omega_k
\end{equation}
the $1$-form $\theta$ is a contact form. Note also that
\begin{equation}\label{exact}
\d\omega+(-1)^{n-1}\,n\,\mu\wedge\theta=0. 
\end{equation}
The geometric significance of these forms is the following: Suppose $f : \Sigma \to M$ is an oriented hypersurface and $\mathcal{G}_f : \Sigma \to U$ its orientation compatible Gauss lift. In other words the value of $\mathcal{G}_f$ at $p \in \Sigma$ is the unique unit vector at $f(p)$ which is $g$-orthogonal to $f^{\prime}(T_p\Sigma)$ and together with a positively oriented basis of $T_p\Sigma$ induces the positive orientation of $T_{f(p)}M$. By construction we have 
\begin{equation}\label{normal}
\mathcal{G}_f^*\theta=0
\end{equation}
and simple computations show that
\begin{equation}\label{legendre}
\mathcal{G}_f^*\omega=\omega_{f^*g},
\end{equation}
where $\omega_{f^*g}$ denotes the Riemannian volume form on $\Sigma$ induced by $f^*g$. Suppose $\tilde{f} : V \subset \Sigma\to F$ is a local framing covering $\mathcal{G}_f$ and $f$. Then pulling back \eqref{normal} and using \eqref{contact} gives 
$$
\sum_{k=1}^n\tilde{f}^*{\omega_{0k}}\wedge \tilde{f}^*{\omega_k}=0.
$$  
The independence \eqref{legendre} implies that the forms $\varepsilon_i=\tilde{f}^*\omega_i$ are linearly independent and thus Cartan's lemma yields the existence of functions $h_{ik} : V \to \R$, symmetric in the indices $i,k$, such that
$$
\tilde{f}^*\omega_{0i}=\sum_{k=1}^n h_{ik} \varepsilon_k. 
$$
In particular we have 
\begin{equation}\label{minimal}
\mathcal{G}_f^*\mu=-H \varepsilon_1\wedge\cdots\wedge\varepsilon_n,
\end{equation}
where $H=\frac{1}{n}\sum_{i=1}^n h_{ii}$ is the mean curvature of the hypersurface $f : \Sigma \to M$. Conversely if $\mathcal{G} : N \to U$ is an orientable $n$-submanifold with $\mathcal{G}^*\theta=0$ and $\mathcal{G}^*\omega \neq 0$, then $\tau \circ \mathcal{G} : N \to M$ is an immersion. Shrinking $N$ if necessary we can assume that $f=\tau \circ \mathcal{G} : N \to M$ is a hypersurface which can be oriented in such a way that its Gauss lift agrees with $\mathcal{G}$. Thus the Legendre integral manifolds $\mathcal{G} : \Sigma \to U$ of the M.A.S. $\mathcal{M}_g$ on $U$ given by
$$
\mathcal{M}_g=\left\{\theta,\d\theta,\mu\right\}
$$
which satisfy the transversality conditions $\mathcal{G}^*\omega\neq0$ locally correspond to minimal hypersurfaces on $(M,g)$. 

\subsection{$\X$-pseudosoliton hypersurfaces via frames} 
Given a vector field $\X$ on $M$ define the functions $X_i : F \to \R$ by 
\begin{equation}\label{comp}
(p,e_0,\ldots,e_n) \mapsto g_p(\X(p),e_i).
\end{equation}
Of course $X_0$ is the $\nu$-pullback of a function on $U$ which will be denoted by $X$. Using \eqref{soli} and \eqref{minimal} it follows that an oriented hypersurface $f : \Sigma \to M$ is a $\X$-pseudosoliton hypersurface if and only if 
$$
\mathcal{G}_f^*\left(\mu-X\omega\right)=0. 
$$ 
Thus the Legendre integral manifolds $\mathcal{G} : \Sigma \to U$ of the M.A.S. $\mathcal{M}_{g,\X}$ on $U$ given by
$$
\mathcal{M}_{g,\X}=\left\{\theta,\d\theta,\mu-X\omega\right\}
$$
which satisfy the transversality conditions $\mathcal{G}^*\omega\neq0$ locally correspond to $\X$-pseudo\-soliton hypersurfaces on $(M,g)$. 
Now suppose $\X$ is a gradient vector field $\X=\nabla_g u$ for some smooth function $u : M \to \R$. Let $\bar{g}=e^{-2u}g$, $\bar{\pi} : \bar{F} \to M$ denote the bundle of positively oriented $\bar{g}$-orthonormal frames with canonical coframing $\bar{\omega}_i, \bar{\omega}_{ik}$ and $\tilde{\psi} : F \to \bar{F}$ the map which scales a $\bar{g}$-orthonormal frame by $e^{u}$. Then by definition
\begin{equation}\label{trafo1}
\tilde{\psi}^*\bar{\omega}_i=e^{-u}\omega_i,
\end{equation}
and the structure equations \eqref{struceq} yield
\begin{equation}\label{trafo2}
\tilde{\psi}^*\bar{\omega}_{ik}=\omega_{ik}+u_k\omega_i-u_i\omega_k,
\end{equation}
where we expand $\pi^*\d u=\sum_{k=0}^n u_k \omega_k$ for some smooth functions $u_k : F \to \R$. Note that $u_0$ is the $\nu$-pullback of the function $X$. Let $\bar{\tau} : \bar{U} \to M$ denote the $\bar{g}$-unit tangent bundle with canonical forms $\bar{\mu},\bar{\omega}$ and $\psi : U \to \bar{U}$ the map which scales a $g$-unit vector by $e^u$. Then \eqref{trafo1} implies
$$
\psi^*\bar{\omega}=e^{-nu}\omega,
$$
thus $\psi$ is a contact diffeomorphism. Moreover \eqref{trafo1} and \eqref{trafo2} yield
\begin{equation}\label{pullback}
\aligned
\psi^*\bar{\mu}&=-\frac{e^{-(n-1)u}}{n}\sum_{k=1}^n (-1)^{k-1}\left(\omega_{0k}+u_k\theta-u_0\omega_k\right)\wedge \hat{\omega}_k\\
&=-e^{-(n-1)u}\left(\mu-X\omega+\frac{1}{n}\theta\wedge\left(i_{\nabla_g u}\omega\right)\right)\\
\endaligned
\end{equation}
which can be written as $\alpha\wedge\theta+\gamma\wedge(\mu-X\omega)$ for some $(n-1)$-form $\alpha$ and some smooth real-valued function $\gamma$ on $U$. This yields 
$$
\psi^*\mathcal{M}_{e^{-2u}g}=\mathcal{M}_{g,\nabla_{g} u}.
$$
Summarising we have proved the
\begin{proposition}\label{solgrad}
Let $(M,g)$ be an oriented Riemannian manifold and $\X=\nabla_g u$ a gradient vector field. Then the $\X$-pseudosoliton M.A.S. on $(M,g)$ is equivalent to the minimal hypersurface M.A.S. on $(M,e^{-2u}g)$. 
\end{proposition}

\subsection{The non-gradient case} Proposition \ref{solgrad} raises the question if there still exists a contact equivalence between minimal hypersurfaces and solitons if $\X$ is not a gradient vector field. We will argue next that this is not possible for $n \geq 2$, so assume in this subsection that $n\geq 2$. Before providing the arguments we recall a result from symplectic linear algebra. Suppose $(V,\Theta)$ is a symplectic vector space of dimension $2n$, i.e.~ $\Theta \in \Lambda^2(V^*)$ is non-degenerate. If a form $\beta$ of degree $s \leq p$  satisfies
\begin{equation}\label{lepage}
\beta\wedge\Theta^{(n-p)}=0,
\end{equation}
then $\beta=0$. This is a corollary of the Lepage decomposition theorem for $p$-forms on symplectic vector spaces. (cf.~\cite[Corollary 15.15]{MR882548}). Of course in our setting the symplectic vector spaces are the fibres of the contact subbundle $D$ and $\Theta$ is obtained by restricting $\d\theta$ to $D$.
\begin{lemma}\label{nec}
A necessary condition for the $\X$-pseudosoliton M.A.S. to be equivalent to the minimal hypersurface M.A.S. is the existence of an exact $1$-form $\rho$ such that 
$$
\d\left(\left(\mu-X\omega\right)\wedge \theta\right)=\rho\wedge\left(\mu-X\omega\right)\wedge\theta
$$ 
\end{lemma}
\begin{proof}
Write $\varphi=\mu-X\omega$ and suppose there exists a Riemannian metric $\bar{g}$ and a diffeomorphism $\psi : U \to \bar{U}$ such that $\psi^*\mathcal{M}_{\bar{g}}=\mathcal{M}_{g,\X}$. Then 
\begin{equation}\label{contactequi} 
\psi^*\bar{\mu}=\alpha\wedge\theta+\beta\wedge \d\theta+\gamma\wedge\varphi,
\end{equation}
where $\alpha$ is a $(n-1)$-form, $\beta$ a $(n-2)$-form and $\gamma$ a smooth real-valued function on $U$. Note that we have
\begin{equation}\label{zero}
\aligned
0&=\varphi \wedge \d\theta,\\
0&=\bar{\mu}\wedge \d\bar{\theta}.
\endaligned
\end{equation}
Wedging \eqref{contactequi} with $\psi^*\d\bar{\theta}$, using \eqref{zero} and that $\psi$ is a contact diffeomorphism gives
\begin{equation}\label{wedge}
\left(\beta\wedge\d\theta\wedge \d\theta\right)\vert_{\D}=0,
\end{equation}
where $\vert_{\D}$ denotes the restriction to the contact subbundle $\D \subset TU$. For $n=2$ equation \eqref{wedge} implies $\beta=0$. For $n \geq 3$ it follows with \eqref{lepage} and \eqref{wedge} that $\beta\vert_{\D}=0$, thus there exists a $(n-3)$-form $\beta^{\prime}$ such that
$$
\beta=\beta^{\prime}\wedge\theta.
$$
We can therefore assume that there exists a $(n-1)$-form $\alpha^{\prime}$ such that 
\begin{equation}\label{contactequi2} 
\psi^*\bar{\mu}=\alpha^{\prime}\wedge\theta+\gamma\wedge\varphi.
\end{equation}
Wedging both sides of \eqref{contactequi2} with $\psi^*\bar{\theta}$ gives
$$
\psi^*\left(\bar{\mu}\wedge\bar{\theta}\right)=\left(\alpha^{\prime}\wedge\theta+\gamma\wedge\varphi\right)\wedge\psi^*\bar{\theta}.
$$
this is equivalent to 
$$
\psi^*\left(\bar{\mu}\wedge\bar{\theta}\right)=\tilde{\gamma}\wedge\varphi\wedge\theta
$$
for some smooth non-vanishing real-valued function $\tilde{\gamma}$. Since $\bar{\mu}\wedge\bar{\theta}$ is an exact form, see \eqref{exact}, we must have 
$$
\d\xi=\d f\wedge\xi,
$$
where we have written $\xi=\varphi\wedge\theta$ and $f=-\ln \vert \tilde{\gamma}\vert$. 
\end{proof}

Using this Lemma we can proof the 
\begin{theorem}\label{main}
The $\X$-pseudosoliton M.A.S. on an oriented Riemannian manifold $(M,g)$ of dimension $n+1\geq 3$ is equivalent to a minimal hypersurface M.A.S. if and only if $\X$ is a gradient vector field. 
\end{theorem}

\begin{remark}
Before giving the proof we point out identities which hold for the functions $X_i$ (recall \eqref{comp} for their definition). Since $O=(\omega_{ik}) \in \mathcal{A}^1(F,\mathfrak{so}(n+1))$ is a connection form we have $O(\W_v)=v$, where $\W_v$ is the vector field obtained by differentiating the flow 
$$
\left((p,e_0,\ldots,e_n),t\right) \mapsto (p,e_0,\ldots,e_n) \cdot \exp(tv)
$$ 
and $v \in \mathfrak{so}(n+1)$, the Lie algebra of $SO(n+1)$. In particular this implies that the time $t$ flow of the vector field $\W_{ik}$ for $i<k$ maps the frame 
$$
(p,e_0,\ldots,e_i,\ldots,e_k,\ldots,e_n)
$$ to the frame
$$
\left(p,e_0,\ldots,\cos (t) e_i-\sin (t)e_k,\ldots, \sin (t)e_i+\cos (t)e_k,\ldots, e_n\right)
$$ 
and thus 
\begin{equation}\label{id}
\mathcal{L}_{\W_{ik}}X_j=\delta_{jk}X_i-\delta_{ij}X_k, 
\end{equation}
where $\mathcal{L}$ stands for the Lie-derivative. 
\end{remark}

\begin{proof}[Proof of Theorem \ref{main}]
We have 
$$
\d X_0=\sum_{i=0}^n P_i \omega_i-\sum_{k=1}^n X_k\omega_{0k}
$$
for some smooth functions $P_i : F \to \R$. From this it follows with straightforward computations that the $1$-forms $\rho$ on $U$ which satisfy $\d\xi=\rho \wedge \xi$ pull-back to $F$ to become
\begin{equation}\label{eq1}
\nu^*\rho=\lambda \omega_0+n\sum_{k=1}^n X_k \omega_k
\end{equation}
for a smooth function $\lambda : F \to \R$. Differentiating \eqref{eq1} gives
$$
\nu^*\d\rho=\d\lambda\wedge\omega_0-\lambda\sum_{k=1}^n\omega_{0k}\wedge\omega_k+n\sum_{k=1}^n\d X_k\wedge\omega_k-n\sum_{i=0}^n\sum_{k=1}^nX_k\omega_{ki}\wedge\omega_i.
$$
Wedging with $\omega_0\wedge\hat{\omega}_1$ yields
$$
\nu^*\d\rho\wedge\omega_0\wedge\hat{\omega}_1=\left(\lambda\, \omega_{01}-n\,\d X_1-n\,\sum_{k=1}^n X_k\,\omega_{1k}\right)\wedge\omega_0\wedge\omega.
$$
Using \eqref{id} we can expand
$$
\aligned
\d X_1\wedge\omega_0\wedge\omega&=\left(\left(\mathcal{L}_{\W_{01}}X_1\right)\omega_{01}+\sum_{k=1}^n\left(\mathcal{L}_{\W_{1k}}X_1\right)\omega_{1k}\right)\wedge\omega_0\wedge\omega\\
&=\left(X_0\,\omega_{01}-\sum_{k=1}^nX_k\,\omega_{1k}\right)\wedge\omega_0\wedge\omega.\\
\endaligned
$$
Concluding we get
$$
\nu^*\d\rho\wedge\omega_0\wedge\hat{\omega}_1=(\lambda-nX_0)\,\omega_{01}\wedge\omega_0\wedge\omega.
$$
Suppose the $\X$-pseudosoliton M.A.S. on $(M,g)$ is equivalent to a minimal hypersurface M.A.S. Then, by Lemma \ref{nec}, $\rho$ has to be exact, this implies
$$
\lambda=nX_0
$$
and thus
$$
\nu^*\rho=n\sum_{i=0}^n X_i \omega_i. 
$$
Note that if $\chi \in TF$ is a vector tangent to the frame $(p,e_0,\ldots,e_n)$ we have 
$$
\sum_{i=0}^n\left(X_i\omega_i\right)(\chi)=\sum_{i=0}^ng_p\left(g_p(\X(p),e_i)e_i,\pi^{\prime}(\chi)\right)=g_p\left(\X(p),\pi^{\prime}(\chi)\right),
$$
hence
$$
\rho=n\,\tau^*\left(\X^{\flat}\right),
$$
where $\X^{\flat}$ denotes the $g$-dual $1$-form to $\X$. The $1$-form $\rho$ is exact and thus $\rho=\d f$ for some real-valued function $f$ on $U$ which is locally constant on the fibres of $\tau : U \to M$. Since the $\tau$-fibres are connected, it follows that $f$ is constant on the $\tau$-fibres and thus equals the pullback of a smooth function $u$ on $M$ for which
$$
du=n\,\X^{\flat}.
$$ 
In other words $\X$ is a gradient vector field. Conversely if $\X$ is a gradient vector field, then the $\X$-pseudosoliton M.A.S. on $(M,g)$ is equivalent to the minimal hypersurface M.A.S. on $(M,e^{-2u}g)$ by Proposition \ref{solgrad}. 
\end{proof}

\begin{remark}
In~\cite{MR1985469}, Bryant, Griffiths and Grossman study the calculus of variations on contact manifolds in the setting of differential systems. In particular they give necessary and sufficient conditions for a M.A.S. to be locally of Euler-Lagrange type, i.e.~locally equivalent to a M.A.S. whose Legendre integral manifolds correspond to the solutions of a variational problem. In fact, if one replaces Lemma \ref{nec} with~\cite[Theorem 1.2]{MR1985469} a proof along the lines of Theorem \ref{main} shows that for $n \geq 2$ the $\X$-pseudosoliton M.A.S. is \textit{locally} equivalent to a M.A.S. of Euler-Lagrange type if and only if $\X$ is a gradient vector field. 
\end{remark}

\subsection{The surface case}

Recall that in the case $n=1$ of a surface $(M,g)$, the solutions of the $\X$-pseudosoliton equation \eqref{soli} are immersed curves on $M$ which may be reparametrised to become geodesics of a Weyl connection. In his Ph.D.~thesis~\cite{mettlerphd}, the second author has constructed a $10$-parameter family of Weyl connections on the $2$-sphere whose geodesics are the great circles, and thus in particular projectively equivalent to the Levi-Civita connection of the standard spherical metric. Inspection shows that there are Weyl connections in this $10$-parameter family whose vector field is not a gradient and thus they provide counterexamples to Theorem \ref{main} in the surface case.

This raises the question what the necessary and sufficient conditions for the $\X$-pseudosolitons curves are, in order to be the geodesics of a Riemannian metric. In~\cite{mettlerphd} it was also shown that on a surface locally every affine torsion-free connection is projectively equivalent to a Weyl connection. Finding the necessary and sufficient conditions thus comes down to finding the necessary and sufficient conditions for an affine torsion-free connection to be projectively equivalent to a Levi-Civita connection. Therefore the conditions follow by applying the results in~\cite{MR2581355} and we refer the reader to this source for further details.       

\providecommand{\bysame}{\leavevmode\hbox to3em{\hrulefill}\thinspace}
\providecommand{\MR}{\relax\ifhmode\unskip\space\fi MR }
\providecommand{\MRhref}[2]{%
  \href{http://www.ams.org/mathscinet-getitem?mr=#1}{#2}
}
\providecommand{\href}[2]{#2}

\flushleft

\small
\vspace{0.5cm}
\sc
Eidgen\"ossische Technische Hochschule Z\"urich, Switzerland\\
\vspace{0.1cm}
\textit{Email address:} \rm \href{mailto:norbert.hungerbuehler@math.ethz.ch}{norbert.hungerbuehler@math.ethz.ch}

\vspace{0.25cm}
\sc
University of California at Berkeley, CA, USA\\
\vspace{0.1cm}
\textit{Email address:} \rm \href{mailto:mettler@math.berkeley.edu}{mettler@math.berkeley.edu}

\end{document}